\documentclass[11pt]{amsart}
\usepackage[hmarginratio=1:1]{geometry}              
\usepackage{graphicx}
\usepackage{amssymb}
\usepackage{epstopdf}
\usepackage{amsthm} 
\usepackage{marvosym}
\usepackage{enumerate}
\usepackage[all]{xy}
\usepackage[linktocpage]{hyperref}
\usepackage{float}

\usepackage{caption}
\usepackage{subcaption}

\newtheorem{thm}{Theorem}
\numberwithin{thm}{section}

\newtheorem{lemma}[thm]{Lemma}

\newtheorem{proposition}[thm]{Proposition}

\newtheorem{corollary}[thm]{Corollary}
\newtheorem{question}[thm]{Question}

\theoremstyle{remark}
\newtheorem{rmk}[thm]{Remark}

\theoremstyle{definition}
\newtheorem{definition}[thm]{Definition}
\newtheorem{example}[thm]{Example}

\newcommand{\overl}[1]{\overline{#1}}

\newcommand{\Z}{\mathbb{Z}}
\newcommand{\Q}{\mathbb{Q}}
\newcommand{\T}{\mathbb{T}}
\newcommand{\E}{\mathbb{E}}
\newcommand{\Sa}{\mathbb{S}}

\newcommand{\Aut}{\text{Aut}}
\newcommand{\Out}{\text{Out}}
\newcommand{\Mod}{\text{Mod}}
\newcommand{\Diff}{\text{Diff}}

\newcommand{\id}{\text{id}}
\newcommand{\what}{\widehat}
\newcommand{\til}{\widetilde}
\newcommand{\IA}{\text{IA}}
\newcommand{\Inn}{\text{Inn}}
\newcommand{\Tor}{\mathcal{I}}
\newcommand{\Gl}{\text{Gl}}

\title{Automorphisms and Homology of Non-positively Curved Cube Complexes}
\author{Corey Bregman}
\begin{document}
\maketitle
\begin{abstract}

We define an integer-valued invariant of special cube complexes called the genus, and prove that having genus one characterizes special cube complexes with abelian fundamental group.  Using the genus, we obtain a new proof that the fundamental group of a special cube complex is either free abelian or surjects onto a non-cyclic free group. We also investigate automorphisms of special cube complexes, and give a new geometric proof that the Torelli subgroup for a right-angled Artin group is torsion-free.

\end{abstract}

\section{Introduction}
Non-positively curved (NPC) cube complexes were introduced by Gromov \cite{Gro87} as a large source of easily constructible examples of locally CAT(0) spaces. These spaces are built by identifying Euclidean $n$-cubes $[-1,1]^n$ along their faces by isometries, subject to certain local combinatorial conditions. Recently, NPC cube complexes have come into prominence in geometric group theory and low-dimensional topology through their r\^{o}le in Agol's solution \cite{Ag13} to the virtually Haken and virtually fibered conjecture for hyperbolic three-manifolds. More generally, NPC cube complexes naturally arise when one considers (relatively) hyperbolic groups which can be built up from the trivial group by iterated amalgamation over (relatively) quasi-convex subgroups (\cite{WiseQCH}, \cite{Ag13}, \cite{AGM16}). 

Haglund and Wise \cite{HaWi08} introduced a restricted class of cube complexes called \emph{special} cube complexes.  \emph{Special groups}, or groups which arise as fundamental groups of finite-dimensional special cube complexes, are known to enjoy many nice properties; in particular, when the cube complex is compact, they embed in SL$(n,\Z)$ and are residually torsion-free nilpotent \cite{HaWi08}. The latter implies moreover that such groups are indicable, \emph{i.e.} they surject onto $\Z$. Both of the above stated properties are consequences of the fact that fundamental groups of compact special cube complexes embed into right-angled Artin groups (raags).  Raags are in some sense the prototypical examples of compact special groups, and are characterized by having presentations in which any two generators either commute or generate $F_2$.  

Let $G$ be a special group.  In this paper we will be interested in how the abelianization $H_1(G)$ determines the geometry of a special cube complex $X$ with $\pi_1(X)\cong G$.  It was shown by Wise \cite{WiseQCH} (see also Koberda--Suciu \cite{KoSu16}) that special groups which are not virtually abelian are \emph{large}; they have finite index subgroups which surject onto the non-abelian free group $F_2$.  In particular, the rank of $H_1$ grows at least linearly after passing to finite index subgroups.  Wise further asked (\cite{WiseQCH}, pg. 143) whether any special group is either abelian or surjects onto a non-cyclic free group.  Our main theorem answers this question in the affirmative
\begin{thm}\label{main} Let $G$ be the fundamental group of a finite dimensional special cube complex.  Then either $G$ is abelian or surjects onto $F_2$.
\end{thm}

Theorem \ref{main} was originally proved by Antol\'{i}n and Minasyan \cite{AM15} who, using different methods, showed that any subgroup of a (finitely or infinitely generated) right-angled Artin group is either abelian or surjects onto $F_2$. In our proof, we introduce an invariant of special cube complexes which we call the \emph{genus}.  This definition has a classical motivation, namely the original combinatorial genus of a surface due to Betti and Poincar\'{e} \cite{PoAS}:  The genus of a closed surface $\Sigma$ is the maximal number of disjoint non-separating simple closed curves whose union does not disconnect $\Sigma$.  Analogously, if $X$ is special then $g(X)$ is defined to be the number of pairwise disjoint, non-separating hyperplanes whose union does not disconnect $X$.  We extend this definition to special groups by defining $g(G)$ to be the maximum genus over all $X$ with $\pi_1(X)=G$ (cf. \S 3).  
Clearly if $g(G)=n$ then $G=\pi_1(X)$ surjects onto $F_n$. The geometric analogue of Theorem \ref{main} characterizes low values of the genus explicitly:
\begin{thm} \label{main1}Let $X$ be special and finite dimensional.  Then
\begin{enumerate}
\item $g(X)=0$ if and only if $X$ is CAT(0).
\item $g(X)=1$ if and only if $\pi_1(X)$ is abelian. 
\item If $\Sigma_g$ denotes the closed orientable surface of genus $g$, then $g(\pi_1(\Sigma_g))=g$. 
\end{enumerate}
\end{thm}

In particular, the classical definition of genus agrees with ours.  The geometric content of this theorem is that if $G$ is special and not abelian, and $X$ is any special cube complex with $\pi_1(X)=G$, there exists a map of cube complexes $X\rightarrow S^1\vee S^1$.  We remark that for a  general group $G$ a notion related to the genus is  the \emph{corank}, i.e. the largest rank of a free group onto which $G$ surjects.  If $G=\pi_1(M)$ for some smooth manifold $M$, then the corank is the same as the \emph{cut number}, the largest number of disjointly embedded, 2-sided hypersurfaces in $M$ whose union does not separate.  This follows from the fact that the wedge of $n$ circles is a $K(F_n,1)$.  Thus the genus of a special group $G$ gives a lower bound for the corank.  It would be interesting to know whether the genus is always equal to the corank.

In the second half of the paper we investigate automorphisms of special groups and the action of the automorphisms of a cube complex on first homology. There are two parts to this problem: (1) which automorphisms of $G$ can be realized as an automorphism of $X$, a compact cube complex with $\pi_1(X)=G$, and (2) when does an automorphism of $X$ act non-trivially on $H_1(X)=H_1(G)$.   Denote by $\Aut(G)$ the group of automorphisms of $G$, $\Out(G)$ the group of outer automorphisms of $G$, and $\Tor(G)\leq \Out(G)$ the subgroup of automorphisms acting trivially on $H_1(G)$.

The motivation for answering these questions comes from classical results on Riemann surfaces and free groups.  Let $\Sigma=\Sigma_g$ be a surface of genus $g\geq2$, and denote by $\Mod(\Sigma)$ its mapping class group, \emph{i.e.} the group of orientation-preserving diffeomorphisms of $\Sigma$ up to homotopy. The Dehn--Nielsen--Baer theorem identifies $\Mod(\Sigma)$ as an index 2 subgroup of $\Out(\pi_1(\Sigma))$. If $\phi\in \Mod(\Sigma)$ has finite order, it is a classical result that there exists a hyperbolic surface $X$ diffeomorphic to $\Sigma$ and an isometry $f:X\rightarrow X$ which \emph{realizes} the homotopy class of $\phi$ (see for example, \cite{FaMa12}).  This fact can be used to show that the Torelli subgroup $\Tor(\pi_1(\Sigma))$ is torsion-free, by showing that any isometry of a constant curvature surface acts non-trivially on first homology.  The theorem reduces an algebraic question about subgroups of the mapping class groups to a geometric question about isometries of a compact surface.  

Similarly, for free groups, Culler \cite{Cul84}, Zimmermann \cite{Zim81}, and Khramtsov \cite{Kh85} each independently showed that any finite order automorphism $\phi \in \Out(F_n)$ can be realized as an automorphism of a simplicial graph $\Gamma$ of rank $n$ . An easy geometric argument then recovers the result of Baumslag-Taylor that $\Tor(F_n)$ is torsion-free for all $n$ \cite{BaTa68}.

Recently, for each raag $A_\Gamma$, Charney, Stambaugh and Vogtmann \cite{CSV12} defined a contractible simplicial complex $K_\Gamma$ on which a subgroup of $\Out(A_\Gamma)$ acts properly discontinuously, cocompactly by simplicial automorphisms.  Their space is defined in analogy with outer space for free groups, and if $A_\Gamma=F_n$, $K_\Gamma$ is just the spine of outer space.  Using $K_\Gamma$, we show
\begin{thm}\label{main2} Let $\phi\in \Out(A_\Gamma)$ have finite order.  Then $\phi$ acts non-trivially on $H_1(A_\Gamma)$. In particular, $\Tor(A_\Gamma)$ is torsion-free.  
\end{thm}
This theorem is originally due to Wade \cite{Wa13}, and independently Toinet \cite{To13}, who proved the stronger result that the Torelli subgroup associated to $\Out(A_\Gamma)$ is residually torsion-free nilpotent. However, both these proofs are almost entirely algebraic, while ours is geometric in the same spirit as those outlined for mapping class groups and free groups above. To prove Theorem \ref{main2}, we first realize $\phi$ as a finite order automorphism of a compact special cube complex $X$ whose fundamental group is $A_\Gamma$, then prove that any such automorphism acts non-trivially on $H_1(X)=H_1(A_\Gamma)$. 

We also include a realization result for automorphisms of special groups which are $\delta$-hyperbolic, and the following result about large groups which is elementary but which we nevertheless could not find in the literature.    
\begin{thm} \label{main3}Suppose $G$ surjects onto $F_2$ and let $\phi\in \Out(G)$ have finite order.  Then there exists a finite index normal subgroup $N\trianglelefteq G$ and an outer automorphism $\psi \in \Out(N)$ such that $\psi$ acts non-trivially on $H_1(N)$ and $\psi_*=\phi_*\circ\iota_*$, where $\iota:N\rightarrow G$ is the inclusion.
\end{thm}
Rephrased in terms of spaces, if $X$ is a $K(G,1)$ and $\phi \in \Out(G)$ we can represent $\phi$ as a homotopy equivalence $f:X\rightarrow X$.  Then there exists a finite regular cover $p:\what{X}\rightarrow X$, a homotopy equivalence $\what{f}:\what{X}\rightarrow \what{X}$ such that $f\circ p= p\circ \what{f}$, and such that $\what{f}_*$ acts non-trivially on $H_1(\what{X})$.  
\\

\noindent \textbf{Outline:} In section 2, we discuss basic facts and terminology concerning NPC cube complexes, hyperplanes, and cohomology.  This is where we introduce one of our main technical tools, namely collapsing hyperplanes.  In section 3, we define the genus of a special cube complex, compute the genus for several examples of groups, and prove Theorem \ref{main1}.  In section 4, we discuss automorphisms of cube complexes, and determine a criterion which guarantees that any automorphism acts non-trivially on first homology.  We also give a proof of Theorem \ref{main3} and some of its applications.  Finally, in section 5 we review the construction of Charney--Stambaugh--Vogtmann's outer space for raags and apply the criterion of section 4 to blow-ups of Salvettis to prove Theorem \ref{main2}.
\\

\noindent \textbf{Acknowledgements:} I would like to thank my advisor, Andrew Putman, for encouraging me and helping me to clarify my ideas. I would also like to thank Jason Manning for several useful discussions and comments, and Ashot Minasyan for explaining how to derive Theorem \ref{main} from his results in \cite{AM15}.  Finally, I thank Letao Zhang and Neil Fullarton helping me through every incarnation of this paper up until now. 

\section{Hyperplanes and cohomology}
Let $X$ be a \emph{non-positively curved (NPC)} cube complex. $X$ is a path-metric space obtained by gluing together standard Euclidean cubes of the form $[-1,1]^n$ by identifying faces by isometries.   For us, $X$ will always be finite dimensional.  Recall that NPC means the universal cover $\widetilde{X}$ of $X$ is a CAT(0) cube complex, and is equivalent to Gromov's condition that the link of each vertex is a flag simplicial complex \cite{Wise12}. If $x\in X^{(0)}$ denote the \emph{link} of $x$ by $lk(x)$.  

We recall that a \emph{midcube} of a cube $C=[-1,1]^n$ is a subset of $C$ obtained by restricting one of the coordinates to 0. A \emph{hyperplane} of $X$ is a maximal connected subset of $X$ which meets each cube in a midcube. Hyperplanes are important subsets of cube complexes and will play a key role in what follows.  If $H\subset X$ and $r>0$, we will use the notation $N_r(H)$ to denote the open $r$-neighborhood of $H$, and $\overl{N_r}(H)$ to denote the closed $r$-neighborhood.  

The universal cover $\widetilde{X}$ comes equipped with both a CAT(0) metric and a combinatorial metric defined on its 0-skeleton, where the distance between two vertices is the number of hyperplanes which separate them.  A \emph{combinatorial geodesic} in $\widetilde{X}$ is a path in the 1-skeleton of $\widetilde{X}$ which crosses each hyperplane at most once (see \cite{Wise12}).  A \emph{combinatorial geodesic} in $X$ is a path in the 1-skeleton of $X$ which lifts to a combinatorial geodesic in $X$.

\begin{definition}A map of cube complexes $f:X\rightarrow Y$ is called a \textbf{local isometry} if the following two conditions hold:
\begin{enumerate}
\item For every $x\in X^{(0)}$, the map $f:lk(x)\rightarrow lk(f(x))$ is injective.
\item If $u,v\in lk(x)^{(0)}$ and $f(u)$ and $f(v)$ are adjacent in $lk(f(x))$, then $u$ and $v$ are adjacent in $lk(x)$.
\end{enumerate}
\end{definition}
\noindent
Local isometries lift to convex embeddings of universal covers: 
\begin{lemma}\label{localIsom}$($\cite{Wise12}, Lemma 3.12 $ )$ If $f:X\rightarrow Y$ is a local isometry, then the induced map on universal covers $\til{f}:\til{X}\rightarrow \til{Y}$ is a convex embedding of CAT(0) spaces.  In particular, $f_*:\pi_1(X)\hookrightarrow \pi_1(Y)$ is an injection. 
\end{lemma}

\subsection{Special cube complexes}
Recall that $X$ is called \emph{special} if $X$ is NPC and that none of the following hyperplane pathologies occur in $X$:
\begin{enumerate}
\item One-sided hyperplanes,
\item Self-intersecting hyperplanes,
\item (Directly) Self-osculating hyperplanes,
\item Interosculating hyperplanes.

\end{enumerate}
For more details, see \cite{HaWi08} and \cite{Wise12}.  In particular, for any hyperplane $H\subset X$, conditions (1) and (2) imply that a $\epsilon$-neighborhood $N_\epsilon(H)$ of $H$ is isometric to a product $H\times[-\epsilon,\epsilon]$ for some $\epsilon>0$ ($\epsilon=1/2$ will do). Choosing an orientation on $H$ we can consistently orient the 1-cubes dual to a midcube of $H$.  We say that two oriented 1-cubes $e$ and $e'$ are \emph{parallel} if they are dual to the same hyperplane with the same orientation, denoted $e\|e'$.  We use square brackets $[e]$ to designate the equivalence class of oriented edges parallel to $e$.  

In the sequel, we will often cut open cube complexes along hyperplanes and consider the resulting cube complex. Define \emph{$X$ split along $H$} to be the complex $X|H$ defined as follows.  $X\setminus N_1(H)$ is a closed subcomplex of $X$, hence is compact special (\cite{HaWi08}, Corollary 3.9) .  There are natural inclusions $\iota^+,\iota^-:H\rightarrow X\setminus N_1(H)$. Note that it may be the case that $X=H\times \Sa^1$ in which case $\iota^+=\iota^-$. Then define \[X|H=H\times[0,2]\coprod X\setminus N_1(H)\coprod H\times[3,5]/\left(H\times \{2\}\sim\iota^-(H), H\times \{3\}\sim\iota^+(H)\right).\]
We denote by $H^-$ the image of $H\times\{0\}$ and by $H^+$ the image of $H\times\{5\}$ under this construction.

\begin{definition} A finitely generated group $G$ is \textbf{(NPC) cubulated} if $G=\pi_1(X)$ for some compact NPC cube complex $X$.  We say further that $G$ is \textbf{compact special} if $X$ is compact special. 
\end{definition}

The prototypical examples of compact special groups are \emph{right-angled Artin groups (raags)}, defined as follows.  
\begin{definition} Let $\Gamma=(V,E)$ be a finite simplicial graph.  If $V=\{v_1,\ldots,v_n\}$, the right-angled Artin group $A_\Gamma$ associated to $\Gamma$ is the group with presentation \[A_\Gamma=\left \langle\begin{array}{l|l} v_1,\ldots v_n & [v_i,v_j],\mbox{ if $v_i$, $v_j$ share and edge in $\Gamma$}\end{array}\right \rangle.\]
\end{definition}

To each raag $A_\Gamma$ is associated a canonical NPC compact special cube complex called the \emph{Salvetti complex} $\Sa_\Gamma$. The Salvetti complex has the following cell structure:
\begin{itemize}
\item $\Sa_\Gamma^{(1)}$: Take a wedge of $n$ circles, one for each vertex $v_1,\ldots, v_n\in V$.
\item $\Sa_\Gamma^{(2)}$: For each edge $(v_i,v_j)\in E$, attach a square $[-1,1]^2$ along $v_iv_jv_i^{-1}v_j^{-1}$.  Its image is a torus $\T^2\subseteq \Sa_\Gamma^{(2)}$.
\item $\Sa_\Gamma^{(k)}$: For each complete $k$-subgraph $K$ of $\Gamma$, attach a $k$-cube $[-1,1]^k$ by identifying its boundary with the $k$-many $(k-1)$-tori in $\Sa_\Gamma^{(k-1)}$ corresponding to complete $(k-1)$-subgraphs of $K$.
\end{itemize}
\noindent
In fact, Salvetti complexes are universal receptors for compact special cube complexes:
\begin{thm} \label{embedSal}$($\cite{Wise12}, Theorem 4.4.$)$ Let $X$ be compact special.  Then there is a Salvetti complex $\Sa_X$ and a local isometry $f_X:X\rightarrow \Sa_X$.  
\end{thm}
\begin{corollary} Every compact special group is the subgroup of a raag.  
\end{corollary}
The corollary follows directly from Lemma \ref{localIsom} above.  The Salvetti complex $\Sa_X$ arises from the raag with defining graph $\Gamma(X)$ equal to the \emph{crossing graph} of $X$:  the vertices of $\Gamma(X)$ are in bijection with the hyperplanes of $X$, and there is an edge between two vertices if their corresponding hyperplanes cross.  

%
%
%
%
\subsection{Collapsing separating hyperplanes}

\begin{definition}
A hyperplane $H$ is \textbf{separating} if $X\setminus H$ has more than one connected component. Otherwise, $H$ is \textbf{non-separating}. 
\end{definition}
%
%
If $H$ is separating then $\overl{N_1}(H)\cong H\times [-1,1]$.  We will now describe a way of collapsing $X$ along these product neighborhoods to obtain a NPC cube complex with no separating hyperplanes. We first learned of the technique of collapsing hyperplanes in \cite{CSV12} and then adapted it to our setting.  
\begin{definition} Let $H\subset X$ be a separating hyperplane, and $\overl{N_1}(H)$ be its closed unit neighborhood.  We define the \textbf{collapse $X/H$ of $X$ along $H$ }to be the cube complex obtained by the identification $\overl{X}_H=X/\{(x,t)\sim (x,s)\}$ where $(x,t), (x,s) \in N_1(H)\cong H\times [-1,1]$. Let $\pi:X\rightarrow X/H$ denote the quotient map.
\end{definition}

\begin{proposition}\label{collapse} If $X$ is special and $H$ is separating, then the collapse $X/H$ is special. 

\end{proposition}
\begin{proof} 

The fact $\overl{N_1}(H)\cong H\times[-1,1]$ implies that $X/H$ has a cube complex structure.  First we check that $X/H$ is NPC.  For this it suffices to check the Gromov link condition at each vertex.  If a vertex does not meet $\overl{N_1}(H)$ then its link passes isometrically to the quotient, hence the link condition is still satisfied.  If a vertex $v_0$ lies in $\overl{N_1}(H)$, then in the quotient $v_0$ is identified with exactly one other vertex $v_1$ which lies at the other end of an edge dual to $H$. Call $e$ the edge joining $v_0$ and $v_1$.

Denote the link of $v_0$ by $lk(v_0)$, and the full subcomplex of $lk(v_0)$ generated by cubes other than $e$ which meet $\overl{N_1}(H)$ by  $lk_H(v_0)$.  Finally denote the full subcomplex generated by cubes in $lk(v_0)$ other than $e$ by $lk_e(v_0)$.  We similarly obtain complexes $lk(v_1)$, $lk_H(v_1)$ and $lk_e(v_1)$. Note that $lk_H(v_i)$ are exactly the edges in link of $v_i$ which lie on the boundary of a cube containing $e$, for $i=0,1$.  If $m$ is the midpoint of $e$, then $m$ is a vertex of $H$ and $lk(m) \cong lk_H(v_0)\cong lk_H(v_1)$.  Since $H$ is NPC, $k_H(v_i)$ is a flag simplicial complex, and hence a full subcomplex of $lk(v_i)$ and $lk_e(v_i)$.  There are no monogons in the quotient because $H$ does not self-intersect, and there are no bigons because $X$ is NPC.  Thus, in the quotient the link of the vertex corresponding to the equivalence class of $v_0,v_1$ can be described as $lk_e(v_0)\coprod lk_e(v_1)$ identified along $lk_H(v_0)\cong lk_H(v_1)$.  This is flag because it is made from two flag complexes glued along a full subcomplex.\\

\noindent
\textbf{No one-sided hyperplanes:} Let $\pm [e_H]$ denote the equivalence class of parallel edges in $X$ which are dual to $H$.  Then in the quotient this class vanishes, and all other classes are preserved. Suppose the $e||-e$ in $X/H$.  Then $e,-e$ are dual to some hyperplane $K$ and there is a path $\gamma$ between the endpoints of $e$ lying entirely within $N_1(K)\setminus K$. Let $H_0$ be the image of $H$ under the collapse $\pi:X\rightarrow X/H$.  Any path which meets $H_0$ has a lift to $X$, since the pre-image of any segment $I$ lying in $H_0$ is a rectangle $I\times[-1,1]$.  Thus, $\pi^{-1}(K)$ is not 2-sided.  \\

\noindent
\textbf{No self-intersection:} Suppose $K$ is a hyperplane in $X$ which intersects itself in $X/H$. Then there are two squares in $X$ with edges $ e, e'$ dual to $K$ connected by an edge $e_0\in [e_H]$. Then these two squares are opposite faces of a cube $C$ containing $e_0$ as a dual edge, and $e,e'$ extend to $C$ to intersect in $C$.  Thus $K$ intersected itself in $X$.  \\

\noindent
\textbf{No self-osculation:} Suppose $K$ is a hyperplane in $X$ which self-osculates in $X/H$. Then there are two edges $e,e'$ in $X$ dual to $K$, and lying on opposite sides of an edge $e_0$ dual to $H$.  If $H$ and $K$ do not intersect, then $H$ is not separating.  If they do intersect, then the fact that $X$ is special implies that they meet in a square in $X$.  Since it is not possible for $K$ to self-intersect or self-osculate, there is a single square bounded on parallel sides by edges dual to $K$ and on the other by $H$.  It follows that after collapsing $H$, $K$ does not self-osculate in the quotient. \\

\noindent
\textbf{No interosculation:} Suppose $K_1$ and $K_2$ are hyperplanes of $X$ which interosculate in the quotient.  Note that as in the case of no-self-intersection, it is not possible for $K_1$ and $K_2$ to intersect in the quotient if they did not in $X$.  Thus, $K_1$ and $K_2$ cross in $X$, and there are a pair of edges $e_1$ and $e_2$, dual to $K_1$ and $K_2$ respectively, which lie at either ends of an edge $e_0$ dual to $H$.   There are three cases, depending on whether or not $K_1$ and $K_2$  intersect or osculate $H$ in $X$.  If both $K_1$ and $K_2$ osculate $H$, then $H$ does not separate.  If exactly one of $K_1$ and $K_2$ intersects $H$, say $K_1$, then $K_1$ and $H$ cross in a square with boundary $e_0$ and $e_1$.  Then $K_1$ and $K_2$ interosculate in $X$, or they cross in a square at the other end of $e_0$.  It follows that under the collapse, no interosculation occurs.  In the case where all three intersect, then in $X$ there is a 3-cube containing $e_0$, $e_1$ and $e_2$ and hence $\pi(K_1)$ and $\pi(K_2)$ cross in $X/H$.\\

\noindent
Since none of the four hyperplane pathologies can occur in the quotient, $X/H$ is NPC and special.  
\end{proof}

\begin{rmk}\label{StillSep} Note that if $K$ and $H$ separate $X$ then $\pi(K)$ still separates in $X/H$.
\end{rmk}

\begin{definition} A special cube complex $X$ is called \textbf{irreducible} if it has no separating hyperplanes.  Otherwise $X$ is \textbf{reducible}.  

\end{definition}

\begin{proposition} \label{Irreducible}Every compact special cube complex $X$ is homotopy equivalent to an irreducible compact special cube complex.   
\end{proposition}
\begin{proof} 
An easy application of van-Kampen's theorem shows that collapsing separating hyperplanes in Proposition \ref{collapse} induces an isomorphism on $\pi_1$.  Since both $X$ and the quotient are NPC, they are each $K(\pi_1,1)$'s, hence homotopy equivalent. By compactness, there are only finitely many separating hyperplanes, and by Remark \ref{StillSep}, we can collapse them in order.  
\end{proof}

\subsection{The cohomology group $H^1(X)$}


Let $X$ be an NPC cube complex and suppose that every hyperplane is embedded and two-sided.  If $H$ is non-separating, then $H$ defines a surjection $\phi_H\colon \pi_1(X)\rightarrow \Z$ as follows.  First, choose an orientation on 1-cubes dual to $H$.  This is possible since $H$ is two-sided.  For each 1-cube $e\in X$ define $\widetilde{\phi_H}(e)$ to be the signed intersection of $e$ with $H$ and extend to 1-chains $C_1(X)$ by linearity. To see that $\widetilde{\phi_H}$ is a cocycle, observe that the signed sum around any square, again by two-sidedness of $H$, is 0. Since $H$ is non-separating, there exists a cycle in $X^{(1)}$ which meets $H$ exactly once with positive orientation. For any $\alpha\in H_1(X)$, we define the intersection product $\alpha.H=\phi_H(\alpha)$. 

Combining the above observation with Proposition \ref{collapse} we can characterize exactly when a special cube complex is CAT(0):
\begin{corollary} \label{CharCAT}Suppose $X$ is connected and special.  Then the following are equivalent\begin{enumerate}
\item$X$ is CAT(0).
\item $H_1(X)=0$. 
\item Every hyperplane is separating.
\end{enumerate}

\end{corollary}
\begin{proof} If $X$ is CAT(0) then $\pi_1(X)$ is trivial and hence $H_1(X)$ is as well.  If $X$ has a non-separating hyperplane then by the observation $H^1(X)$ is non-trivial. Finally, suppose every hyperplane is separating. Any compact subset $K\subset X$  is contained in the closed unit neighborhoods of only finitely many hyperplanes hence Proposition \ref{collapse} and Remark \ref{StillSep} imply that $K$ can be collapsed to a point.  In particular, $\pi_1(X)$ is trivial and hence $X$ is CAT(0).
\end{proof}
As a final corollary, we have the following curious observation about quasiconvex hierarchies for hyperbolic special groups. 
\begin{corollary} If $G$ is $\delta$-hyperbolic and $G=\pi_1(X)$ for some compact special cube complex, then $G$ has a quasiconvex hierarchy consisting only of HNN-extensions.
\end{corollary}
\begin{proof} Let $H$ be a non-separating hyperplane.  Then $\pi_1(H)$ is quasiconvex in $G$, and $G \cong \pi_1(X|H)*_{\pi_1(H)}$.  Now collapse separating hyperplanes in $X|H$ and repeat.  Note that $H^+$ and $H^-$ are separating in $X/H$. Eventually we will end up with a complex which only has separating hyperplanes, since the total number of cubes decreases every time we split along hyperplanes and collapse.  
\end{proof}

\section{The genus of a special group}
As we saw in the previous section, each non-separating hyperplane of a special cube complex $X$ contributes a free factor to $H^1(X)$, but in general these free factors may not be distinct.  For example, if $K_1$ and $K_2$ are two disjoint non-separating hyperplanes such that $K_1\cup K_2$ separates $X$, then every homology class which meets $K_1$ also meets $K_2$ and with the same algebraic intersection, hence $\phi_{K_1}=\phi_{K_2}$.  Based on this observation we have the following
\begin{definition} Let $X$ be special.  The \textbf{genus} $g(X)$ is the maximum number of disjoint hyperplanes in $X$ whose union does not separate. If no maximum exists we say $g(X)=\infty$.  If $\Gamma$ is the fundamental group of a special cube complex, we define the genus \[g(\Gamma)=\sup\{g(X)\colon X\mbox{ is special and $\pi_1(X)=\Gamma$}\}.\]
\end{definition}

The definition is motivated by the classical definition genus of compact surface: namely, the largest number of disjoint simple closed curves whose union does not disconnect the surface. The next proposition lists some properties of the genus.  

\begin{proposition} \label{gProps} Let $X$ be a special cube complex (resp. special group). The genus enjoys the following properties:
\begin{enumerate}
\item $g(X)\leq \text{rk}(H_1(X))$. In particular, $g(X)$ is finite whenever $X$ is compact (resp. finitely generated).  
\item $g(X)=0$ if and only if $X$ is CAT(0) (resp. $X=\{1\}$).

\end{enumerate} 
\end{proposition}
\begin{proof} Let $X$ be a special cube complex.  If $K_1$ and $K_2$ are disjoint and do not separate, there are homology classes $\gamma_1$ and $\gamma_2$ in $H_1(X)$ such that $K_i. \gamma_j=\delta_{ij}$ for $i,j=1,2$.  Hence the $K_i$ correspond to distinct free factors of $H_1(X)$.  This proves (1).  Property (2) follows directly from Corollary \ref{CharCAT}.  
\end{proof}
\noindent
We calculate the genus of some basic examples of special groups:

\begin{example} $g(F_n)=n$.  Take the standard rose $R_n$ as a cube complex with $\pi_1=F_n$.  Then $g(R_n)=1$ and property (1) implies this is best possible.  In fact, any graph with $\pi_1=F_n$ works.
\end{example}

\begin{example} $g(\Z^n)=1$.  This follows from the fact that if $g(\Gamma)=n$ then $\Gamma$ surjects onto $F_n$.  
\end{example}

\begin{example} $g(\pi_1(\Sigma_g))=g$, where $\Sigma_g$ is the closed surface of genus $g$.  Note that this is not entirely obvious from the definition, since there may well be high-dimensional cube complexes with the same fundamental group as $\Sigma_g$.  We observe, however, that if $K_1,\ldots, K_n$ are disjoint hyperplanes of $X$, then $\phi_{K_i}\cup\phi_{K_j}=0\in H^2(X)$ for all $i,j$.  Hence, the $\phi_{K_i}$ span a Lagrangian subspace of $H^1(X)=H^1(\Sigma_g)$.  The maximal possible dimension of such a subspace is of course $g$, and it is not hard to construct explicit 2-D cube complex structures on $\Sigma_g$ which realize this maximum. 
\end{example}

\begin{definition} Hyperplanes $K$ and $L$ are called \textbf{parallel} if $K$ and $L$ do not meet.  
\end{definition}

\begin{thm}\label{genus1}Let $X$ be compact special.  Then $g(X)=1$ if and only if $\pi_1(X)=\Z^n$.
\end{thm}
\begin{proof}

Without loss of generality, we can assume $X$ is irreducible by Proposition \ref{Irreducible}, since collapsing separating hyperplanes does not change the genus.  Let $K$ be a hyperplane of $X$.  Since $X$ has genus one, we know that if $L$ is any hyperplane parallel to $K$, then $K\cup L$ separates. The idea is to imitate the proof of Proposition \ref{collapse} by collapsing all hyperplanes parallel to $K$, while maintaining non-positive curvature and specialness. Define an equivalence relation on hyperplanes as follows.  $K\sim L$ if $K$ and $L$ are parallel. As defined, this is just a symmetric relation, but we will consider the equivalence relation $\sim^*$ that it generates and say that if $K\sim^* L$ then $K$ and $L$ are \emph{ultra-parallel}.  If $K$ and $L$ are ultra-parallel and $X$ has genus one, then every minimal combinatorial loop which meets $K$ once also meets $L$ once.  Moreover, we note that if $K$ and $L$ are ultra-parallel but not parallel, then $K$ and $L$ intersect.  

Let $K_1,\ldots, K_m$ be the set of hyperplanes other than $H$ which are ultra-parallel to $H$. We want to show that result of collapsing all of $K_1,\ldots, K_m$ is NPC and special.  We remark that it may be the case that some subset of collapses fails to be special. The important point is that the full collapse is special and homotopy equivalent to $X$.  First we need a lemma which implies that we can collapse hyperplanes at all.

\begin{lemma} \label{productNbhd} If $K, L\subset X$ are distinct parallel hyperplanes, then $\overl{N_1}(K)\cong K\times [-1,1]$ and $\overl{N_1}(L)\cong L\times [-1,1]$.
\end{lemma}
\begin{proof} The lemma is symmetric in $K$ and $L$.  If $\overl{N_1}(K)$ is not embedded, then $K\cup L$ does not separate.  
\end{proof}

From the lemma, if $L$ is a hyperplane in an ultra-parallelism class $[H]$, $H\neq L$, we know that $\overl{N_1}(L)$ is embedded.  A slight issue arises when we repeatedly collapse hyperplanes.  Namely, if we collapse all the hyperplanes in $[H]$ which are actually parallel to $L$, then $\overl{N_1}(L)$ ceases to be embedded.  However, by the next lemma, we can always collapse in such a way that every hyperplane neighborhood is embedded.

\begin{lemma}\label{CollapseSequence}Given an equivalence class $[H]$ with $|[H]|\geq 2$, there always exists a sequence of collapses in which, at each stage, the remaining hyperplanes have embedded unit neighborhoods.  
\end{lemma}
\begin{proof} Given a collection of hyperplanes $\mathcal{H}$ in $X$, let $\Delta(\mathcal{H})$ be the graph obtained in the following way.  The vertices of $\Delta(\mathcal{H})$ will be the elements of $\mathcal{H}$, and two vertices are connected if their corresponding hyperplanes are disjoint.  If $\mathcal{H}=[H]$ is an ultra-parallelism class, then $\Delta([H])$ is connected. Suppose $X'$ is obtained from $X$ by collapsing a hyperplane $L$ and let $[H]'$ be the image of $[H]$ in $X'$.  Observe that $\Delta([H]')$ can be obtained from $\Delta([H])$ by deleting the vertex corresponding to $L$, and all of its incident edges.  In particular, if $L\notin[H]$, then $\Delta([H])=\Delta([H]')$.  

The lemma, translated in terms of $\Delta([H])$, states that there is a sequence of vertex deletions such that, at each stage the complement is connected.  The latter follows by induction and the well-known graph theoretic result that a connected graph always has at least two vertices which are not cut vertices, i.e. they do not disconnect the graph.  
\end{proof}

From the two previous lemmas, we know that if $[H]$ is an ultra-parallelism class, we can find $H\in [H]$ and an ordering of the hyperplanes $K_1,\ldots, K_m\in [H]\setminus\{H\}$ such that when we collapse each $K_i$ in order, the result at each stage will be an NPC cube complex homotopy equivalent to $X$.  We remark that as before, two hyperplanes cannot cross in the quotient if they did not originally.  The proof that in the quotient every hyperplane is two-sided and that no hyperplane self-intersects is exactly the same as above, and we do not need the hypothesis that $X$ has genus one.  The next lemma implies no self-osculation occurs in the quotient.  

%
\begin{lemma}\label{otherSide}Fix an ultra-parallelism class $[H]$ which has cardinality at least 2.  Then for all $L\in [H]$, and for every two vertices $v_1$ and $v_2$ in $\overl{N_1}(L)$ lying on the same side of $L$, there does not exist a combinatorial geodesic from $v_1$ to $v_2$ which crosses some edge dual to a hyperplane parallel to $L$.
\end{lemma}
\begin{proof} We claim that either $g(X)\geq 2$ or no such path exists for any $L\in [H]$.  Define $\mathcal{S}$ to be the set of combinatorial geodesics $\gamma$ such that there exists some $L\in [H]$ and satisfying
\begin{enumerate}
\item $\gamma$ does not cross $L$
\item The endpoints of $\gamma$ lie on the same side of some hyperplane $L$
\item $\gamma$ crosses some hyperplane $K$ parallel to $L$ 
\end{enumerate}
We will show that if $g(X)=1$ then $\mathcal{S}$ is empty, by induction on a least length counterexample. For reasons of parity, we consider two base cases, when the length $l(\gamma)=1$ and $l(\gamma)=2$. Let $v_1$ and $v_2$ be the endpoints of $\gamma$.  If $l(\gamma)=1$ then by condition (3), the single hyperplane $K$ which $\gamma$ crosses must be parallel to $L$. Hence, we can complete $\gamma$ to a loop $\gamma'$ by adding a path between $v_1$ and $v_2$ in $N_1(L)$ which meets $K$ exactly once.  Since $L$ is non-separating, we conclude that $g(X)\geq 2$, a contradiction.  If $l(\gamma)=2$, then $\gamma$ crosses two hyperplanes $K_1$ and $K_2$.  If $K_1\neq K_2$, then as in the previous case, since one of $K_1$ and $K_2$ is parallel to $L$, we conclude that $g(X)\geq 2$. If $K_1=K_2$, then $K_1$ is parallel to $L$ and we must consider the orientations with which $\gamma$ crosses $K_1$. If $\gamma$ crosses $K_1$ with the same orientation each time, then we again conclude that $g(X)\geq 2$.  If, on the other hand, $\gamma$ crosses $K_1$ first with one orientation, then the opposite we invoke the fact that $K_1$ does not directly self-osculate to conclude that $\gamma$ backtracks, and hence is not a combinatorial geodesic. 


Now suppose $\gamma$ is a least length counterexample of length $n\geq 3$ occurring along a hyperplane $L\in [H]$.  Let $v_1$ and $v_2$ be the endpoints of $\gamma$.  Then $\gamma$ crosses a sequence of edges dual to hyperplanes $K_{i_1}^{\epsilon_1},\ldots, K_{i_r}^{\epsilon_r}$, where $\epsilon_j=\pm1$ depending on the orientation with which $\gamma$ crosses $K_{i_j}$. \\

\noindent
\textbf{Claim 1:} $L$ is parallel to $K_{i_1}$. 

\noindent
Otherwise, by no interosculation, there is a square with corner $v_1$ where $K_{i_1}$ and $L$ cross.  Then both endpoints of the first edge $e_1$ of $\gamma$ lie in $N_1(L)$. Writing $\gamma=e_1\gamma'$, we see that $\gamma'$ is a shorter length counterexample.\\

\noindent
\textbf{Claim 2:} $\gamma$ crosses $K_{i_1}$ algebraically (i.e. counted with sign) 0 times. 

\noindent
If not, then we can complete $\gamma$ to a loop $\gamma'$ as above, which crosses $K_{i_1}$ algebraically non-zero times and crosses $L$ geometrically 0 times.  Since $L$ is assumed non-separating, we conclude that $g(X)\geq 2$, a contradiction.\\

At this point, we can assume that $\gamma$ must cross $K_{i_1}$ algebraically 0 times.  If we consider the sequence of crossings we can find an innermost pair with opposite sign, i.e. a subpath $e_1\alpha e_2\subset \gamma$ such that $e_1$ and $e_2$ are both dual to $K_{i_1}$ but with opposite orientation, and such that $\alpha$ does not cross $K_{i_1}$.  Clearly $\alpha$ is not empty, otherwise $\gamma$ would have backtracking.  

We claim that either $\gamma$ is not a combinatorial geodesic, or $\alpha$ is a shorter length counterexample.  If $\alpha$ crosses some hyperplane parallel to $K_{i_1}$ then $\alpha\in \mathcal{S}$, since it connects two vertices on the same side $K_{i_1}$, is a combinatorial geodesic since it is a subpath of $\gamma$ and satisfies $l(\alpha)\leq l(\gamma)-2$.  It is therefore a shorter element of $\mathcal{S}$, contradicting our assumption on $\gamma$.  Otherwise, $K_{i_1}$ meets every hyperplane crossed by $\alpha$.  In this case however, we can replace $\gamma$ by a combinatorially isotopic path with backtracking.  To see this, consider the first edge $f_1$ of $\alpha$.  By no interosculation we can find a square $C$ bounded at a corner by $e_1$ and $f_1$.  We therefore replace $e_1\alpha$ by $\alpha'=f_1'e_1'\alpha''$, where $e_1'$ and $f_1'$ are the opposite edges of $C$, and $\alpha''$ is the remainder of $\alpha$ after $f_1$.  Continuing in this way, we replace $e_1\alpha e_2$ by $\alpha_0 e_0\overl{e}_0$, where $e_0$ is dual to $K_{i_1}$. Since $\gamma$ was isotopic to a path with backtracking, we conclude that $\gamma$ was not combinatorially geodesic, contradicting our assumption.  Therefore $\mathcal{S}$ is empty, as desired. 
%
%
%
%
\end{proof}
\noindent
\textbf{No self-osculation:} Suppose that a hyperplane $L$ directly self-osculates after collapsing some collection of the $K_i$.  Then there is a path $\gamma$ consisting of edges dual to some subcollection $K_1,\ldots, K_m$, which connects two vertices lying on the same side of $\overl{N_1}(L)$.  If $L$ intersects each of $K_1,\ldots, K_m$, then no self-osculation occurs in the quotient.  Otherwise, $L$ is parallel to some $K_j$.  But then Lemma \ref{otherSide} implies that this is impossible.  \\

At this point we have checked that after collapsing each of the $K_i$, the resulting space immerses in a Salvetti complex.  For a local isometry, we need to further check that no interosculation occurs.  \\

\noindent
\textbf{No interosculation:} By the remark about intersecting hyperplanes above, we need only consider the case where hyperplanes $L_1$ and $L_2$ intersect in $X$ and osculate in the quotient.  In this case there is a path $\gamma$ dual to hyperplanes $K_1,\ldots,K_m$ which are ultra-parallel to $H_1$ and edges $f_1$ and $f_2$ dual to $L_1$ and $L_2$, respectively at either end of $\gamma$.  Moreover, $L_1$ and $L_2$ meet in some other square. If at least one of $L_1$ and $L_2$ intersects all of the $K_i$, then by no interosculation of $X$, after collapsing there is a square containing $f_1$ and $f_2$.   Finally, we have the case where both $L_1$ and $L_2$ are parallel to one of the $K_i$.  There are three cases depending on which sides of $L_1$ and $L_2$ that $\gamma$ connects. See Figure \ref{fig:Genus1Inter} for a schematic.  Note that in this case all hyperplanes have embedded closed unit neighborhoods.  

At most one of $L_1$ and $L_2$ is the chosen hyperplane $H_1$.  In the case that neither $L_1$ nor $L_2$ is $H_1$ then $L_1$ and $L_2$ are both eventually collapsed and no interosculation occurs in the quotient.  Then assume that $L_1=H_1$. In either case (1), (2), or (3) we find that no interosculation occurs in the quotient and either $H_1$ directly self-osculates, which we have already shown is impossible, or $H_1$ indirectly osculates which does not contradict specialness of the quotient.  \\


\begin{figure*}[h!]

\centering
\begin{subfigure}[t]{0.5\textwidth}
\includegraphics[width=2in]{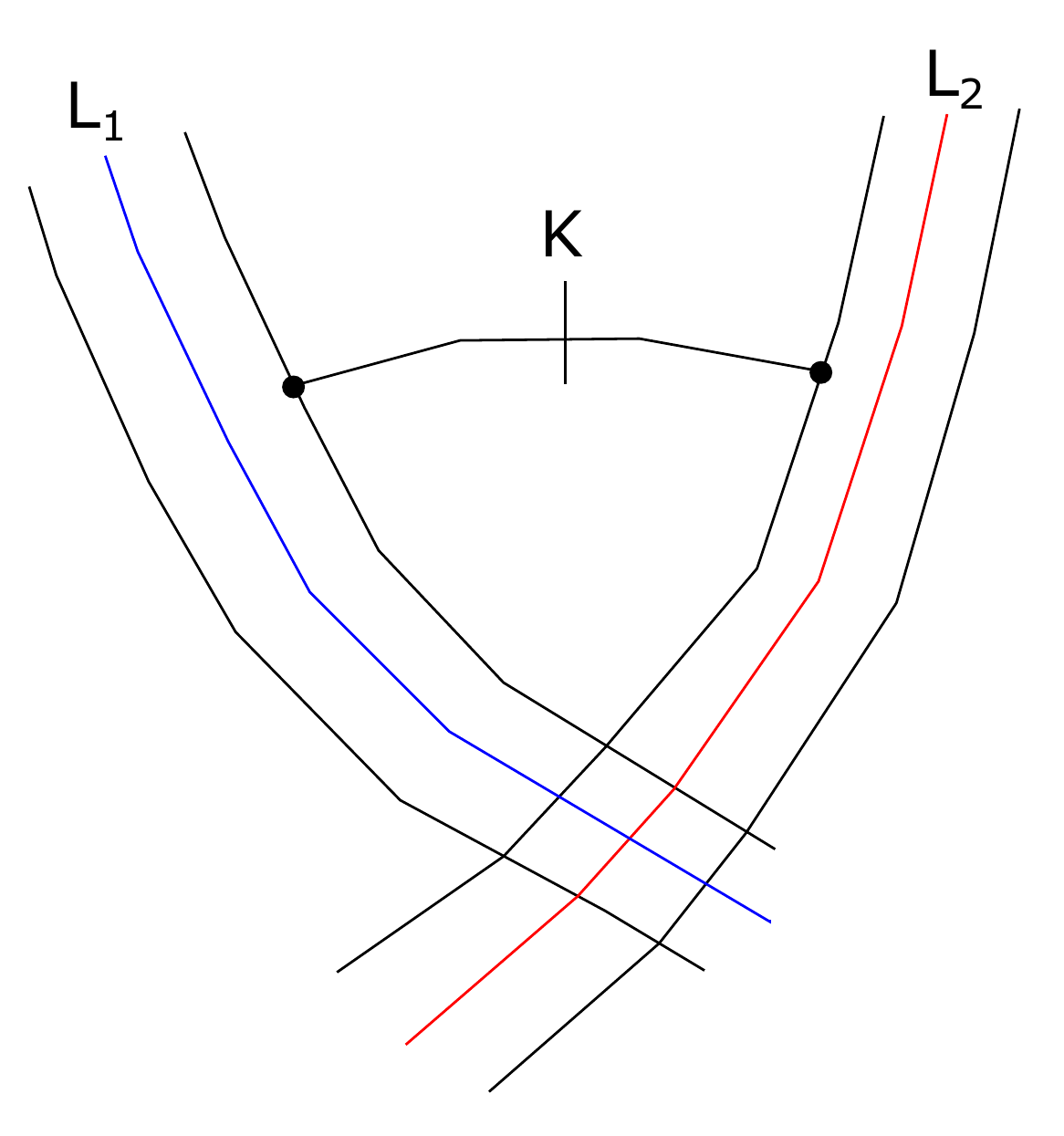}
\centering
\caption*{Case (1)}
\label{Interosc1}
\end{subfigure}
\quad
\begin{subfigure}[t]{0.5\textwidth}
\centering
\includegraphics[width=2in]{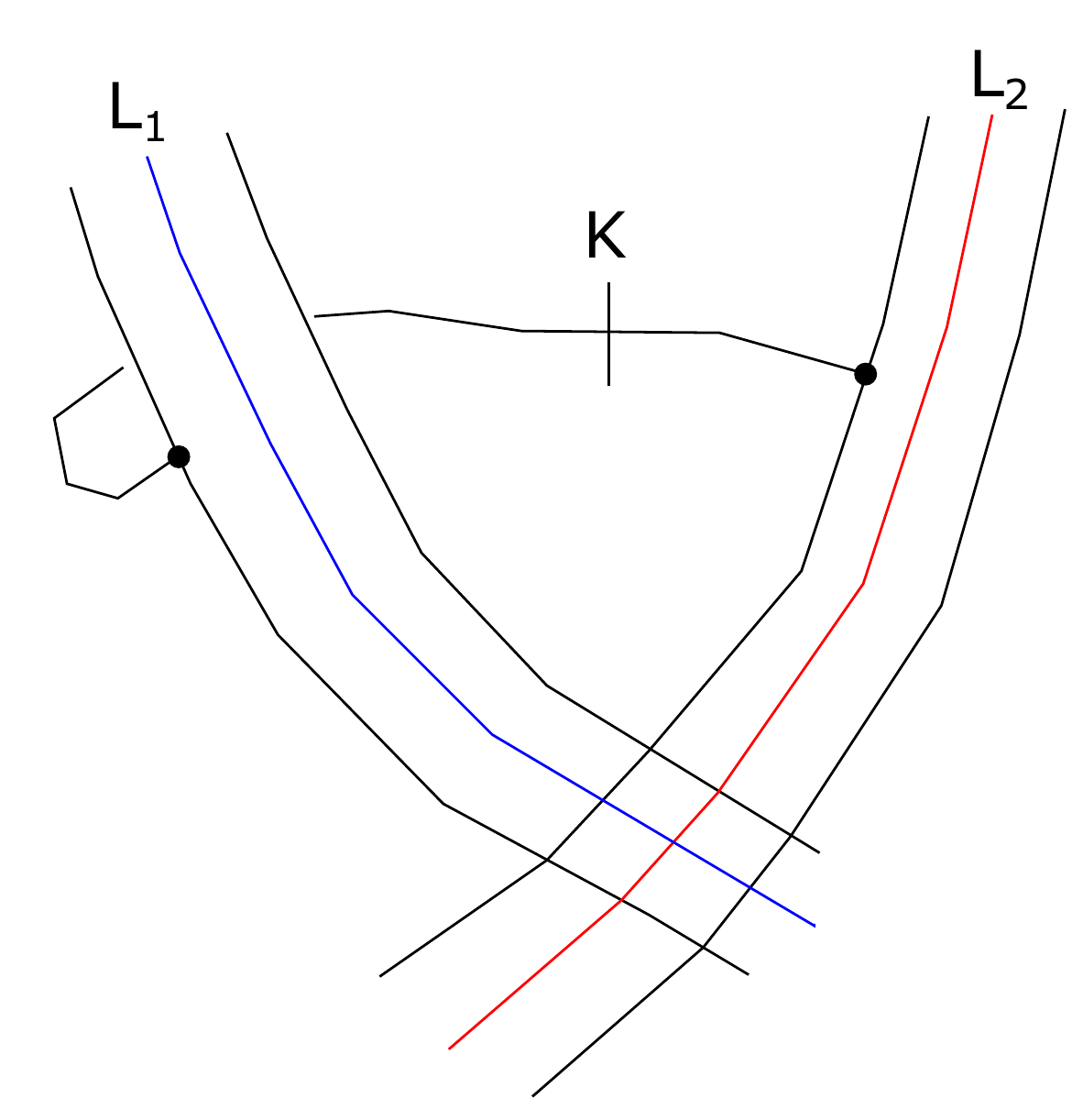}
\caption*{Case (2)}
\label{Interosc2}
\end{subfigure}
\quad
\begin{subfigure}[t]{0.5\textwidth}
\centering
\includegraphics[width=2in]{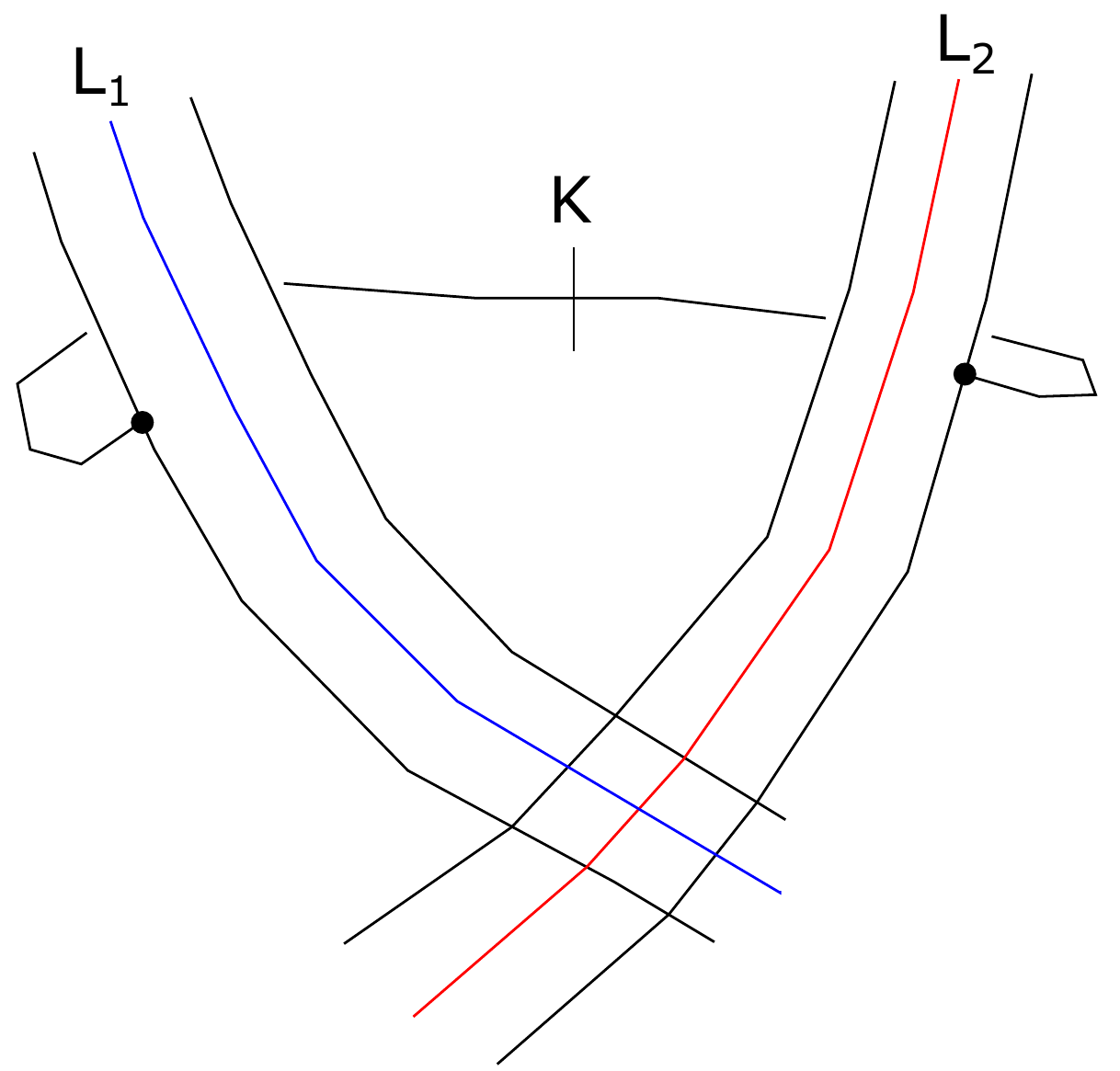}
\caption*{Case (3)}
\label{Interosc3}
\end{subfigure}
\caption{Possible configurations resulting in interosculation in the quotient.}
\label{fig:Genus1Inter}
\end{figure*}



Let $X'$ denote the result of collapsing all hyperplanes ultra-parallel to $H$, this is still NPC by Lemma \ref{CollapseSequence}. To see that $X'$ has genus one, note that if $\mathcal{H}$ is the set of hyperplanes of $X$, then $\mathcal{H}\setminus\{K_i\}$ is in one-to-one correspondence with the set of hyperplanes of $X'$ and that if $L_1\notin \{K_i\}$ and $L_2\notin\{K_i\}$ are disjoint and separate $X$, then their images under the collapse are disjoint and separate $X'$. In $X'$, the image of $H$ meets every other hyperplane. 


To finish the proof, we proceed as follows.  Choose some ultra-parallelism class $[H]$. If $[H]$ is a singleton, then $H$ meets every hyperplane of $X$ and we don't do anything.  If $|[H]|\geq 2$, then by Lemma \ref{CollapseSequence} we can find $H\in [H]$ such that if $K_1,\ldots, K_m$ are hyperplanes ultra-parallel to $H$, we can collapse the $K_i$ in some order such that the resulting cube complex $X'$ is NPC special, has genus one, and every hyperplane of $X'$ meets the image of $H$.  Now choose another ultra-parallelism class and repeat.  Since $X$ has only finitely many hyperplanes, in the end we obtain a special cube complex $Y$ homotopy equivalent to $X$ with the property that any two hyperplanes meet.  It follows that the corresponding Salvetti $\Sa_Y$ is a torus, and hence $\pi_1(X)=\pi_1(Y)\leq \Z^n$ for some $n$, by Lemma \ref{embedSal}.  In fact, since $Y$ is compact, the map $f_Y:Y\rightarrow \Sa_Y$ is a surjective, combinatorial local isometry, hence must be a finite covering. 
\end{proof}

\begin{rmk} If $X$ is non-compact but has finitely many hyperplanes, the same proof as above works. However, it may be the case that the quotient complex $Y$ is non-compact.  Then the characteristic map  $f_Y:Y\rightarrow \Sa_Y$ will be a surjective, combinatorial local isometry, but all we can conclude is that it is a covering map.  Thus the image of $\pi_1(Y)\hookrightarrow\pi_1(\Sa_Y)\cong\Z^n$ may be a subgroup of infinite index.
\end{rmk}

\noindent
As an immediate corollary we obtain:

\begin{corollary}\label{unsplit} If $\Gamma$ is non-abelian then every compact special cube complex with $\pi_1(X)=\Gamma$ satisfies $g(X)\geq2$. 
\end{corollary}

\noindent
We can extend this result to the finite-dimensional case to give a proof of Theorem \ref{main}:
\begin{corollary}\label{Large} If $X$ is special and finite-dimensional then either $\pi_1(X)$ is abelian or surjects onto $F_2$.

\end{corollary}
\begin{proof} The compact case follows from Theorem \ref{genus1} .  In the non-compact case, we observe that if there are infinitely many non-separating hyperplanes, then since $X$ is finite-dimensional there are infinitely many disjoint non-separating hyperplanes.  Take a loop $\gamma$ which meets some non-separating hyperplane exactly once.  Then $\gamma$ only intersects finitely many non-separating hyperplanes.  Thus we must have at least two disjoint non-separating hyperplanes whose union does not disconnect $X$ and $g(X)\geq 2$.  Otherwise, there are only finitely many non-separating hyperplanes and if $g(X)=1$, we can apply the procedure of Theorem \ref{genus1} to these hyperplanes.  We obtain a homotopy equivalent special cube complex $Y$ in which all non-separating hyperplanes meet.  

Suppose $\gamma_1$ and $\gamma_2$ are two loops in the 1-skeleton $Y^{(1)}$ based a point $p\in Y$.  Choose a compact subcomplex $K$ containing $\gamma_1\cup \gamma_2$.  By collapsing all the separating hyperplanes which meet $K$, we get a complex $Y'$ homotopy equivalent to $Y$, and in which the images $\gamma_1'$ and $\gamma_2'$, of $\gamma_1$ and $\gamma_2$ respectively, only cross non-separating hyperplanes.  Since the non-separating hyperplanes of $Y'$ all cross we conclude that the homotopy classes $[\gamma_1']$ and $[\gamma_2']$ commute in $\pi_1(Y')$. We conclude that $\pi_1(Y)=\pi_1(X)$ is abelian.    
\end{proof}

%

The characterization of genus 1 special groups also has some immediate corollaries for groups with genus $\geq 2$.

\begin{corollary} Let $\Gamma$ be a non-abelian special group.  Then \begin{enumerate}
\item $\text{rk}(H_1(\Gamma))\geq2$.  
\item $\Gamma$ retracts onto $F_2$.
\item The rank of $H_1$ grows at least linearly in finite index subgroups.  
\item The growth of finite index subgroups in $\Gamma$ is at least superexponential $(\succeq ne^{n\log(n)-n})$ in index.  
\end{enumerate}
Moreover, if $\Gamma$ is not virtually abelian but is virtually special, $(3)$ and $(4)$ still hold.  
\end{corollary}
\begin{proof} Statements (1) and (2) follow from Corollary \ref{Large} and the fact that any surjection onto a free group splits.  Statements (3) and (4) follow from the corresponding result for $F_2$. The growth of finite index subgroups in a free group is due to Hall \cite{Hall49}.   
\end{proof}

Note that from (2) we get a quick proof of the Tits alternative for virtually special groups: either $\Gamma$ contains a non-abelian free group, or it is virtually abelian.  The genus also restricts which groups can arise as fundamental groups of special cube complexes.  Recall that if $\Sigma_g$ is the closed surface of genus $g$, the group of orientation preserving diffeomorphisms of $\Sigma_g$ is denoted $\Diff^+(\Sigma_g)$.  The mapping class group $\Mod_g$ is defined to be $\pi_0(\Diff^+(\Sigma_g))$, the group of connected components of $\Diff^+(\Sigma_g)$. Thus, two diffeomorphisms are identified if they are isotopic.  The abelianization map $\pi_1(\Sigma_g)\rightarrow \Z^{2g}$ induces a surjective map $\Mod_g\rightarrow \text{Sp}_{2g}(\Z)$. We say $\phi\in \Mod_g$ has \emph{full rank} if the action of $\phi$ on $H_1(\Sigma_g)$ has finite quotient.  


A theorem of Thurston and Nielsen states that every mapping class $\phi\in \Mod_g$ falls into one of three categories: finite order, reducible, or pseudo-Anosov (see \cite{FaMa12}, Theorem 13.2).  \emph{Reducible} means that $\phi$ has a lift to $\Diff^+$ which fixes some 1-submanifold.  \emph{Pseudo-Anosov} means that $\phi$ does not preserve any conjugacy class in $\pi_1(\Sigma_g)$.  Finite order and reducible are not mutually exclusive, but both are disjoint from pseudo-Anosov. Thurston showed that if $\phi$ is pseudo-Anosov, then the mapping torus $M_\phi$ corresponding to any lift of $\phi$ to $\Diff^+$ supports a constant curvature-(--1) Riemannian metric.  This construction provides many examples of hyperbolic 3-manifolds.  

\begin{corollary} If $\Gamma$ is any one of the following, then $\Gamma$ is virtually compact special but not compact special:
\begin{enumerate}
\item virtually abelian but not abelian.
\item the fundamental group of a hyperbolic $\Q H\Sa^3$.  
\item $\pi_1(M_\phi)$ where $\phi$ is pseudo-Anosov and has full rank.   
\end{enumerate}
\end{corollary}
\begin{proof} 

\begin{enumerate}
\item If $\Gamma$ is virtually abelian and compact special then it does not contain $F_2$, hence $g(\Gamma)=1$.  Theorem \ref{genus1} then implies that $\Gamma\cong\Z^n$ for some $n$.  
\item Results of Agol (\cite{Ag13}, Theorems 1.1 and 9.3)  imply that every closed hyperbolic 3-manifold group is virtually compact special.  A hyperbolic rational homology 3-sphere has $\text{rk}(H_1)=0$, hence cannot be compact special. We remark that it was already known that $\text{rk}(H_1)\geq1$ for compact special cube complexes.  
\item A straightforward application of Mayer-Vietoris shows that $\text{rk}(H_1(M_\phi))=1$.  Thurston's theorem and Agol's theorem show that $\pi_1(M_\phi)$ is virtually compact special, but obviously $\pi_1(M_\phi)$ contains $F_2$ and hence is not virtually abelian.  
\end{enumerate}
\end{proof}

\section{Automorphisms of special cube complexes}
In this section we investigate the action of cube complex automorphisms on homology.  Our starting point comes from two well-known examples:\begin{enumerate}
\item Every non-identity torsion element of $\Out(F_n)$ acts nontrivially on $H_1(F_n)$.
\item Every non-identity torsion element of $\Mod_g$ acts nontrivially on $H_1(\Sigma_g)$.
\end{enumerate}
Observe that both $F_n$ and $\pi_1(\Sigma_g)$ are compact special.  This motivates the following
\begin{question}\label{trivOut} Suppose $\Gamma$ is compact special.  Does every non-identity finite order element of $\Out(\Gamma)$ act non-trivially on $H_1(\Gamma)$?
\end{question}
The first step is to ensure that $H_1(\Gamma)$ is non-trivial, but as we have seen, this is satisfied as soon as $\Gamma\neq1$. We do not propose to answer this question fully here, but we will generalize the results above to cubulated hyperbolic and right-angled Artin groups.   Our strategy will be two-fold. First, realize elements of $\Out(\Gamma)$ as automorphisms of compact special cube complexes with fundamental group $\Gamma$.  Second, use the geometry of the compact special cube complex to show certain automorphisms act non-trivially on homology.  A question closely related to the one above is thus
\begin{question}When does an automorphism of a compact special cube complex $X$ act non-trivially on $H_1(X)$?  
\end{question}
Since there are compact CAT(0) cube complexes with arbitrarily large (finite) automorphism groups, the answer to this question is not, in general, ``Always." Nevertheless, we will provide circumstances under which every automorphism acts non-trivially on first homology.
\subsection{A criterion for non-triviality}
The following proposition gives a useful criterion to guarantee that every automorphism acts non-trivially on $H_1(X)$.

\begin{proposition} \label{Criterion}Let $X$ is compact special $f:X\rightarrow X$ is an automorphism.  Suppose $X$ satisfies the following three conditions for hyperplanes $K_1$ and $K_2$
\begin{enumerate}
\item If $K_1\cap K_2\neq \emptyset$, there exists $\alpha\in H_1(X)$ such that $\alpha.K_1\neq \alpha.K_2$.  
\item If $K_1\cap K_2=\emptyset$ and $K_1\cup K_2$ separate $X$, every component of $X\setminus(K_1\cup K_2)$ contains a non-separating hyperplane which does not meet $K_1$ or $K_2$.
\item If $f(K_1)=K_1$ for all $K_1$ then $f$ is the identity.  
\end{enumerate}
Then if $f$ acts trivially on $H_1(X)$, $f$ is the identity.
\end{proposition}
\begin{proof}
Let $f:X\rightarrow X$ be an automorphism and suppose $f_*:H_1(X)\rightarrow H_1(X)$ is the identity.  Assume for contradiction that $f$ is not the identity. The order of $f$ is finite, and by passing to a power we may assume it is a prime $p$.  By condition (1) we know that for any hyperplane $K_0$, the image $f(K_0)$ does not meet $K_0$ transversely, or else $f_*$ would not be the identity. Then for every hyperplane $K_0$, we have that $f(K_0)\cap K_0=\emptyset$ or $f(K_0)=K_0$.  Note that it is not possible for every hyperplane to be mapped to itself setwise without being the identity.  Therefore, either $f=\id_X$ and we are done or there exists some hyperplane $K_0$ such that $f(K_0)\cap K_0=\emptyset$.

As we observed, the images $K_0=f^0(K_0),K_1=f(K_0),\ldots, K_{p-1}=f^{p-1}(K_0)$ are all disjoint and $f$ permutes the components of $X\setminus \cup_{i=0}^{p-1}N_1(K_i)$. We can assume that every cycle which meets $K_0$ algebraically non-trivially also meets each $K_i$ with the same intersection.  In particular, since $K_0$ is non-separating, there is a cycle which meets $K_0$ geometrically once. We conclude that any pair $K_i,K_j$ with $0\leq i\neq j\leq p-1$ separate X. We may also assume, after reordering and passing to a power, that for each $i$, $0\leq i\leq p-1$, one of the components of $X\setminus( K_i\cup K_{i+1})$ does not contain any $K_j$, where $j\neq i, i+1$ and $K_p=K_0$.  For if $K_j$ lies in some component $Y$ of $X\setminus( K_i\cup K_{i+1})$, then it does not meet $K_i$ or $K_{i+1}$ and it must separate $K_i$ from $K_{i+1}$.  Otherwise there is a path in $X$ which meets all of the $K_i$ except $K_j$ geometrically once.  Then we can assume that the $K_i$ are permuted in order, and that one component of $X\setminus( K_i\cup K_{i+1})$ does not contain any other $K_j$.  

By condition (2), we know that each component of $X\setminus( K_i\cup K_{i+1})$ contains a non-separating hyperplane which does not meet $K_i$ or $K_{i+1}$. Choose one such hyperplane $L_i$ in the component of $X\setminus( K_i\cup K_{i+1})$ which does not contain any other $K_j$.  Then since $L_i$ does not meet the $K_i$ or $K_{i+1}$, it is a hyperplane of $X$ proper.  Since it does not separate, there is a non-trivial cycle $\alpha_i$ contained in this component satisfying $\alpha_i.L_i=1$. Then by naturality of the Kronecker pairing \[1=\alpha_i.L_i=\phi_{L_i}(\alpha_i)=f^*(\phi_{L_i})(f_*(\alpha_i))=\phi_{f(L_i)}(\alpha_i)=\alpha_i.f(L_i)=0.\]  
This contradiction implies that $f$ takes every hyperplane to itself, hence must be the identity by condition (3). 
\end{proof}

\subsection{Passing to covers} Even if we cannot guarantee that every automorphism acts non-trivially on $H_1(X)$, in some cases it may be possible to pass to a cover and lift our automorphism so that it acts non-trivially on the homology of the cover. In fact, this is always the case.  The result follows from the next proposition, which although not difficult to prove, does not seem to appear anywhere in the literature.  We record it here for posterity. 

\begin{proposition} \label{LargeAut}Let $G$ any finitely generated group which surjects onto $F_2$, and let $\phi\in \Out(G)$ have finite order. If $\phi_*$ is the induced map on the abelianization $H_1(G)$, then there exists a finite index normal subgroup $N\trianglelefteq G$ and an induced outer automorphism $\what{\phi}$ of $N$ such that $\what{\phi}_*$ does not act trivially on $H_1(N)$.  
\end{proposition}
\begin{rmk} Informally, if $G$ is large, any finite order outer automorphism acts non-trivially on the abelianization of some finite index subgroup.  
\end{rmk}
\begin{proof} Let $\pi:G\rightarrow F_2$ be a surjection.  Since $G$ is finitely generated, $H_1(G)$ is a finitely generated abelian group say with first Betti number $b_1$.  Choose $d>>b_1$ and find some finite-index normal subgroup $K'<F_2$ such that rk$(H_1(K))=d$.  Then $K=\pi^{-1}(K')$ is normal and of finite index in $G$.  Finally let $f$ be a lift of $\phi$ to $\Aut(G)$ and define $N=K\cap\cdots f^{n-1}(K)$ where $n$ is the order of $\phi$.  It follows that $f(N)=N$, hence $f$ induces an automorphism $\what{f}:N\rightarrow N$. We claim that after postcomposing $f$ with congjugation by an element of $G$, the induced map $\what{f}$ not act trivially on $H_1(N)$.

This is just a straightforward application of the transfer homomorphism.  If $\what{f}_*$ acts non-trivially on $H_1(N)$, we are done.  Otherwise, by the transfer, since rk$(H_1(N))=d>b_1$, there exists $g\in G$ such that conjugation by $g$ acts non-trivially on $H_1(N)$.  Denote by $c_g$ the automorphism of $N$ induced by conjugation by $g$.  Then $\what{f}'=c_g\circ \what{f}$ acts non-trivially on $H_1(N)$.  Note that $\what{f}'$ also has finite order.  Setting $\what{\phi}=[\what{f}']\in \Out(N)$ finishes the proof.
\end{proof}

%
\begin{corollary}\label{Cover} Suppose $G$ is finitely generated and virtually compact special.  Then any finite order outer automorphism of $G$ has a lift which acts non-trivially on the abelianization of some finite index subgroup.  
\end{corollary}
\begin{proof} If $G$ is not virtually abelian, then $G$ virtually surjects onto $F_2$ and the result follows by Proposition \ref{LargeAut}.  If $G$ is virtually abelian, then $G$ contains $\Z^n$ as a finite index normal subgroup, for some $n$.  In this case it is not hard to show that rk$(H_1(G))<n$.  Thus the same proof as in Proposition \ref{LargeAut} works here, too.
\end{proof}

In fact, when $G$ is $\delta$-hyperbolic we can say a little more.  Let $X$ be a compact NPC cube complex with $\Gamma=\pi_1(X)$ $\delta$-hyperbolic.  By Theorem 1.1. of \cite{Ag13}, $X$ is virtually compact special.  Note that the center $Z(\Gamma)$ is trivial.
Let $\phi: \Gamma\rightarrow \Gamma$ be a finite-order outer automorphism of $\Gamma$.  From the exact sequence \[1\rightarrow \Gamma\rightarrow \Aut(\Gamma)\rightarrow \Out(\Gamma)\rightarrow 1,\]
we can consider the extension given by pulling back the subgroup $\langle \phi\rangle$: \[1\rightarrow \Gamma\rightarrow  E_\phi\rightarrow \langle\phi\rangle\rightarrow 1.\]
Since $\Gamma$ is cubulated hyperbolic, $E_\phi$ is virtually cubulated hyperbolic. Hence by Lemma 7.15 of \cite{Wise12},  we obtain a proper cocompact action of $E_\phi$ on a CAT(0) cube complex $\til{Y}$.  Since $\Gamma$ is torsion-free, the action of $\Gamma<E_\phi$ on $\til{Y}$ is free.  We therefore obtain a quotient $Y$ with $\pi_1(Y)\cong \Gamma$, and a finite order automorphism $f:Y\rightarrow Y$ corresponding to $\phi$.  We have just proven
\begin{proposition}If $\Gamma$ is cubulated and hyperbolic, every finite order element of $\Out(\Gamma)$ can be realized as an automorphism of an NPC cube complex $Y$ with $\pi_1(Y)=\Gamma$. 
\end{proposition}

\begin{example} Let $K$ be an amphichiral knot such as the figure 8. If $M=\Sa^3\setminus K$ is the knot complement then there is an orientation-preserving involution $\sigma:M\rightarrow M$, induced by the amphichirality.  If $T\subset M$ is a boundary parallel torus, then $\sigma_{|_{T^2}}$ is just the hyperelliptic involution on $T^2$.  In particular, $\sigma$ sends every slope $p/q$ to $-p/-q$.  Let $M_{p/q}$ be the result of $p/q$-surgery on $K$. Since $M_{p/q}=M_{-p/-q}$, the action of $\sigma$ on $M$ extends to an involution $\what{\sigma}:M_{p/q}\rightarrow M_{p/q}$.  If $K$ is hyperbolic (\emph{e.g.} the figure 8 knot), then a theorem of Thurston states that for all but finitely many slopes $M_{p/q}$ will be hyperbolic.  Taking $p=1$, a Mayer-Vietoris computation implies that $M_{1/q}$ will be an integral homology sphere, and will also be hyperbolic for infinitely many $q$.  Then $\what{\sigma}$ is an automorphism of $\pi_1(M_{1/q})$ which is not inner since it has finite order, and non-trivial since it inverts the meridian and longitude of the knot.  Moreover, as $\what{\sigma}$ is orientation-preserving, $\what{\sigma}_*:H_*(M_{1/q})\rightarrow H_*(M_{1/q})$ is actually the identity.  

This construction gives infinitely many $\Z H\Sa^3$'s whose fundamental groups have non-trivial outer automorphism groups.  By Agol's theorem, all of these virtually compact special, implying that Corollary \ref{Cover} is best possible.
\end{example}

\section{Applications to right-angled Artin groups}
In the next two sections we present applications of the results of the previous section to automorphisms of cubulated groups. The first concerns automorphisms of right-angled Artin groups.  

Let $\Gamma=(V,E)$ be a finite simplicial graph, with vertex set $V$ and edge set $E$, and let $A_\Gamma$ be the associated right-angled Artin group.  
%
%
If $V=\{v_1,\ldots, v_n\}$ then $V^{\pm}=\{v_1^{\pm},\ldots, v_n^{\pm}\}$ is a generating set for $A_\Gamma$ with the standard presentation, and the abelianization of $A_\Gamma$ is $A_\Gamma^{ab}\cong \Z^n$.  The abelianization map $\psi:A_\Gamma\rightarrow\Z^n$ induces $\Psi:\Aut(A_\Gamma)\rightarrow \Gl_n(\Z)$ and we obtain short exact sequences
\[1\rightarrow \IA(A_\Gamma)\rightarrow \Aut(A_\Gamma)\xrightarrow{\Psi} \Gl_n(\Z)\]\[1\rightarrow \mathcal{I}(A_\Gamma)\rightarrow \Out(A_\Gamma)\xrightarrow{\overl{\Psi}} \Gl_n(\Z).\]
The kernel $\mathcal{I}(A_\Gamma)=\ker\overl{\Psi}$ (resp. $\IA(A_\Gamma)=\ker\Psi)$ is called the \emph{Torelli subgroup} of $\Out(A_\Gamma)$ (resp. $\Aut(A_\Gamma)$).  $\IA(A_\Gamma)$ and $\mathcal{I}(A_\Gamma)$ are further related by the short exact sequence \[1\rightarrow \Inn(A_\Gamma)\rightarrow \IA(A_\Gamma) \rightarrow\mathcal{I}(A_\Gamma)\rightarrow 1\] where $\Inn(A_\Gamma)\cong A_\Gamma/Z(A_\Gamma)$ is the group of inner automorphisms of $A_\Gamma$. 

The main goal of this section is to prove
\begin{thm} \label{torsionfree}$($Wade \cite{Wa13}, Toinet \cite{To13} $)$ $\mathcal{I}(A_\Gamma)$ is torsion-free for all $\Gamma$.  
\end{thm}
Here we present a geometric proof of this theorem using NPC cube complexes and the machinery developed in the previous section.  Before discussing the strategy of the proof, we list some immediate corollaries.

\begin{corollary}$\IA(A_\Gamma)$ is torsion-free.
\end{corollary}

\begin{corollary}$($Charney--Vogtmann \cite{ChVo09}$)$ $\Out(A_\Gamma)$ and $\Aut(A_\Gamma)$ are both virtually torsion-free.  In particular, each have finite virtual cohomological dimension.  
\end{corollary}
Both of these corollaries are straightforward consequences of the exact sequences above, Theorem \ref{Contractible} below, and Selberg's lemma.  
We will prove the theorem in two steps. We assume for contradiction that $\phi\in \mathcal{I}(A_\Gamma)$ is torsion.  Then 
\begin{enumerate}
\item Realize $\phi$ as an automorphism $f:X\rightarrow X$ of some NPC cube complex $X$ with $\pi_1(X)=A_\Gamma$. This means in particular that the induced map $f_*=\phi$ as an automorphism of $\pi_1(X)$. 
\item Show that for any such $X$, every automorphism acts non-trivially on $H_1(X).$
\end{enumerate}
To carry out Step (1), we will make use of a contractible simplicial complex $K_\Gamma$ on which $\mathcal{I}(A_\Gamma)$ acts.  For Step (2), we will verify the criterion of Proposition \ref{Criterion} for certain NPC cube complexes. 
\subsection{Automorphisms of Raags}

For each $v\in V$ we define two subsets of $V$: \[lk(v)=\{w\in V|w\mbox{ is adjacent to } v\}\] \[ st(v)=\{v\}\cup lk(v).\]  Following \cite{CSV12}, the relation $lk(v)\leq st(w)$ for $v,w \in V$ will be denoted $v\leq w$.  In this case we say $w$ \emph{dominates} $v$. If $v\leq w$ and $w\leq v$ then we write $v\sim w$, in which case $v$ and $w$ are said to be \emph{equivalent}.   

Laurence \cite{Lau95} and Servatius \cite{Ser89} proved that the following four types of automorphisms generate $\Aut(A_\Gamma)$:
\begin{enumerate}
\item \emph{Inversions}: If $v\in V^{\pm}$, the automorphism $i_v$ sends $v\mapsto v^{-1}$ and fixing all other generators.  
\item \emph{Graph Automorphism}: Any automorphism of $\Gamma$ induces a permutation of $V^{\pm}$ which extends to an automorphism of $A_\Gamma$.
\item \emph{Transvections}: If $v\leq w$, the automorphism $\tau_{w,v}$ sends $v\mapsto vw$ and fixes all other generators. 
\item \emph{Partial Conjugations}: If $C$ is a connected component of $\Gamma\setminus st(v)$ for some $v\in V$, the automorphism $\sigma_{v,C}$ maps $w\mapsto vwv^{-1}$ for every $w\in C$, and acts as the identity elsewhere.  
\end{enumerate}

If, in (3), $v$ and $w$ are adjacent, $\tau_{w,v}$ is called an \emph{adjacent} transvection.  Otherwise, $\tau_{w,v}$ is called a \emph{non-adjacent transvection}. As in \cite{Day09} and \cite{CSV12}, we distinguish the subgroup of \emph{long-range automorphisms} $\Out_{\ell}(A_\Gamma)\subseteq \Out(A_\Gamma)$ (resp. $\Aut_{\ell}(A_\Gamma)\subseteq \Aut(A_\Gamma)$)  generated by automorphisms of type (1), (2), (4) and non-adjacent transvections.  

\subsection{The $\Out_{\ell}$-spine $K_\Gamma$}
Recall the definition of the standard Salvetti complex $\Sa=\Sa_\Gamma$ associated to $A_\Gamma$.  $\Sa$ is the cube complex constructed as follows. Start with a single vertex $x_0$. For every $v\in \Gamma$, attach both ends of a 1-cube $e_v$ to $x_0$.  For every complete $k$-subgraph of $\Gamma$, we add in a $k$-cube $C$ whose image is a $k$-torus with 1-skeleton the edges labelled by elements in the subgraph.  $\Sa$ is an NPC cube complex whose fundamental group is $A_\Gamma$.  In particular, $\Sa$ is a $K(A_\Gamma,1)$.  

In \cite{CSV12}, Charney, Stambaugh and Vogtmann constructed a contractible simplicial complex $K_\Gamma$ on which $\Out_{\ell}(A_\Gamma)$ acts properly discontinuously cocompactly by simplicial automorphisms.  Like outer space for free groups, one considers pairs $(X,\rho)$, where $X$ is an NPC cube complex with fundamental group $A_\Gamma$, and $\rho$ is a homotopy equivalence $\rho:X\rightarrow \Sa$.   The pair $(X,\rho)$ is called a \emph{marked blow-up of a Salvetti complex}.  Construction of such cube complexes will be described below.   An automorphism $\phi\in \Out_{\ell}(A_\Gamma)$ acts on $(X,\rho)\in K_\Gamma$ by changing the marking: Represent $\phi$ as a homotopy equivalence $h:\Sa\rightarrow\Sa$.  Then $\phi.(X,\rho)=(X,h\circ\rho)$.   We have 

\begin{thm}\label{Contractible}$($\cite{CSV12}, Propositions 4.17, Theorem 5.24$)$ $K_\Gamma$ is contractible and the action of $Out_{\ell}(A_\Gamma)$ on $K_\Gamma$ is properly discontinuous. 
\end{thm}
See \cite{CSV12} for details on the construction of $K_\Gamma$.  For us, what will be important is that it is contractible, finite dimensional, and admits a properly discontinuous action of $\Out_{\ell}(A_\Gamma)$.


\subsection{The Torelli Subgroup for a Raag}
Day has shown in \cite{Day09} that $\IA(A_\Gamma)$ is generated by automorphisms of the following two forms:
\begin{enumerate}
\item (\emph{Partial Conjugation}) Let $v\in V^\pm$ be a generator, and $C\neq\emptyset$ a component of $\Gamma\setminus st(v)$.  Then $\sigma_{v,C}:A_\Gamma\rightarrow A_\Gamma$ denotes the automorphism 
\begin{align*}
	\sigma_{v,C}: &w\mapsto vwv^{-1}, \mbox{ $w\in C^\pm$;}\\
			     &w\mapsto w, \mbox{ else.}
\end{align*}
\item (\emph{Commutator Transvection}) Let $v, w_1, w_2\in V^\pm$ such that $w_1, w_2$ both dominate $v$, \emph{i.e.} $lk(v)\subset lk(w_1)$, $lk(w_2)$.  Then there are non-adjacent transvections of $v$ by $w_1$ and $w_2$, and we can therefore transvect $v$ by the commutator $[w_1,w_2]$.  

$\tau_{w_1,w_2,v}:A_\Gamma\rightarrow A_\Gamma$ denotes the automorphism
\begin{align*}
	\tau_{w_1,w_2,v}: &w\mapsto [w_1,w_2]w, \mbox{ $w=v$;}\\
			     &w\mapsto w, \mbox{ else.}
\end{align*}

\end{enumerate}

We remark that in case (1), if $\Gamma\setminus st(v)$ is connected, then $\sigma_{v,C}$ is just conjugation by $v$.  From this generating set, it is clear that $\IA(A_\Gamma)\leq \Aut_{\ell}(A_\Gamma)$, since partial conjugations lie in $\Aut_{\ell}(A_\Gamma)$ by definition, and in order for the transvection in case (2) to be non-trivial, we must have that $v$, $w_1$ and $w_2$ are pairwise non-adjacent.  Passing to the outer automorphism group, we obtain $\mathcal{I}(A_\Gamma) \leq \Out_{\ell}(A_\Gamma)$.

Since $\mathcal{I}(A_\Gamma) \leq \Out_{\ell}(A_\Gamma)$, it follows that $\mathcal{I}(A_\Gamma)$ acts on $K_\Gamma$ by simplicial automorphisms.  Suppose $\phi\in \mathcal{I}(A_\Gamma)$ has finite order.  Without loss we may assume the order is prime.  $K_\Gamma$ is finite-dimensional and by Theorem \ref{Contractible} it is contractible, hence the action of $\phi$ on $K_\Gamma$ has a fixed point.  The fixed point $(X,\rho) \in K_\Gamma$ corresponds to a marked blow-up of a Salvetti complex.  By the definition of $K_\Gamma$, this means that $\phi$ is realized as an automorphism $f:X\rightarrow X$ which commutes with the marking $\rho$ up to homotopy.  Further, as $\phi\in \mathcal{I}(A_\Gamma)$, we know that the induced map $f_*:H_1(X)\rightarrow H_1(X)$ is the identity map.  This will be the starting point for our investigation.  We want to show that $f$ itself must be the identity.  We record the preceding discussion in 

\begin{proposition}\label{FixedSal}  Let $\phi\in \mathcal{I}(A_\Gamma)$ have prime order.  Then $\phi$ is realized as an automorphism $f:X\rightarrow X$ of some marked blow-up $X$ of a Salvetti complex.  
\end{proposition}

\subsection{Blow-ups of Salvetti complexes}
At this point, we have realized torsion elements of $\mathcal{I}(A_\Gamma)$ as automorphisms of cube complexes which act trivially on first homology. In order to show that $\mathcal{I}(A_\Gamma)$ is torsion-free, it suffices to show that any automorphism of a blow-up $X$ acts non-trivially on $H_1(X)$. To do this, we will show that every blow-up satisfies the hypotheses of Proposition \ref{Criterion}. 

Generalizing Whitehead partitions for free groups, Charney, Stambaugh, and Vogtmann define automorphisms of $A_\Gamma$ which they call $\Gamma$-Whitehead partitions.  The reason for using a generating set consisting of $\Gamma$-Whitehead automorphisms instead of the standard generating set is that each $\Gamma$-Whitehead automorphism can be achieved by an expansion and collapse of a Salvetti complex for $A_\Gamma$.
\begin{definition}(\cite{CSV12}, Definition 2.1) Let $P\subset V^{\pm}$ have at least 2 elements, including some $m\in P$ with $m^{-1}\notin P$.  Then $(P,m)$ is a \textbf{$\Gamma$-Whitehead pair} if\begin{enumerate}
\item no element of $P$ is adjacent to $m$,
\item if $v\in P$ and $v^{-1}\notin P$ then  $v\leq m$,
\item if $v^{\pm}\in P$, then $w^{\pm}\in P$ for every $w$ in the same component of $\Gamma\setminus st(m)$ as $v$.
\end{enumerate}
\end{definition}

A $\Gamma$-Whitehead pair $(P,m)$ defines an automorphism $\phi=\phi_{(P,m)}$ defined by
\[\phi(v)=\left\{\begin{array}{ll}m^{-1}&\mbox{ if $v=m$}\\
vm^{-1} & \mbox{ if $v\in P$ and $v^{-1}\notin P$}\\
mv & \mbox{ if $v^{-1}\in P$ and $v\notin P$}\\
mvm^{-1} & \mbox{ if $v^{\pm}\in P$}\\
v & \mbox{ else}
\end{array}\right.\]

The pair $(P,m)$ also defines several important subsets of $V^{\pm}$\[double(P)=\{v\in P|v^{\pm}\in P\}\]\[single(P)=\{v\in P|v^{-1}\notin P\}\]\[max(P)=\{v\in single(P)|v\sim m\}\]\[lk(P)=lk(m)^{\pm}\]
The automorphism $\phi$ is clearly a product of an inversion of $m$, a transvection of elements of elements of $single(P)$ and a partial conjugation of elements of $double(P)$, hence $\phi\in Out_{\ell}(A_\Gamma)$.  Conversely, it is easy to see that $Out_{\ell}(A_\Gamma)$ is generated by all $\Gamma$-Whitehead automorphisms together with inversions and graph automorphisms.  Note that each $\phi_{(P,m)}$ has order 2.   

Condition (1) in the definition implies that $P\cap lk(P)=\emptyset$.  The other side of $P$, denoted $P^*$, is the complement of $P\cup lk(P)$ in $V^{\pm}$. $(P^*,m^{-1})$ is also a $\Gamma$-Whitehead pair which defines the same \emph{outer} automorphism of $A_\Gamma$.  We therefore obtain a disjoint union \[V^{\pm}=P\cup lk(P)\cup P^*.\]
\begin{definition} (\cite{CSV12}, Definition 2.4) The triple $\textbf{P}=\{P,lk(P), P^*\} $ is called a \textbf{$\Gamma$-Whitehead partition} of $V^{\pm}$.  $P$ and $P^*$ are the \textbf{sides} of \textbf{P}.
\end{definition}

\begin{definition} (\cite{CSV12}, Definition 3.3) Let $\textbf{P}=\{P,lk(P), P^*\}$ and $\textbf{Q}=\{Q,lk(Q), Q^*\}$ be two $\Gamma$-Whitehead partitions.  \begin{enumerate}
\item \textbf{P}, \textbf{Q} \textbf{commute} if $max(P)$, $max(Q)$ are distinct and commute.
\item \textbf{P}, \textbf{Q} are \textbf{compatible} if either they commute or at least one of $P\cap Q$,  $P^*\cap Q$, $P\cap Q^*$ or $P^*\cap Q^*$ is empty.
\end{enumerate}
\end{definition}

It is shown in \cite{CSV12} that if $\textbf{P}, \textbf{Q}$ are compatible and do not commute, exactly one of the intersections is empty.  A collection $\mathbf{\Pi}=\{\textbf{P}_1,\ldots,\textbf{P}_k\}$ is called \emph{compatible} if the $\textbf{P}_i$ are pairwise compatible.  A \emph{region} of $\mathbf{\Pi}$ is choice of side $P_i^{\times}\in \{P_i,P_i^*\}$ for each $i$, such that for any $i,j$, either $\textbf{P}_i$ and $\textbf{P}_j$ commute, or $P_i^{\times}\cap P_j^{\times}\neq \emptyset$. 

We are now in a position to build the blow-up $\Sa_{\mathbf{\Pi}}$ associated to $\mathbf{\Pi}$.  First we will construct a contractible complex $\E_\Pi$ containing all the vertices of $\Sa_{\mathbf{\Pi}}$. To each region $R=P_1^{\times}\cap\cdots\cap P_k^{\times}$ we associate a vertex $x_R=(a_1,\ldots,a_k)$ of the $k$-cube $[0,1]^k$ via \[a_i=\left\{\begin{array}{ll} 0&\mbox{ if $P_i^{\times}=P_i$}\\
1 & \mbox{ if $P_i^{\times}=P_i^*$}\end{array}\right.\]

Now we attach edges to $\E_\Pi^{(0)}$.  If $R$ and $R'$ are two regions which differ exactly by switching sides along a single partition $P_i$, we attach an edge $e_{P_i}$ from $x_R$ to $x_{R'}$.  The edge $e_{P_i}$ is oriented from the region containing $P_i$ to the region containing $P_i^*$.  The rest of $\E_\Pi$ is formed by filling in cubes where their boundaries occur.  

We complete the construction of $\Sa_{\mathbf{\Pi}}$ by attaching cubes to $\E_\Pi$, starting with the 1-cubes.  Set $\overl{P}_i^{\times}=P_i^{\times}\cup lk(P_i)$.  For each region $R$, define a subset $V^{\pm}$ \[I(R)=\overl{P}_1^{\times}\cap\cdots\cap\overl{P}_k^{\times}.\] Compatibility implies each $I(R)$ is non-empty, and Lemma 3.10 $(1)$ of \cite{CSV12} states that ever $v\in V^{\pm}$ occurs in some $I(R)$. If $v^{\pm}\in I(R)$, attach both vertices of an edge $e_v$ at $x_R$.  Suppose $v\in I(R)$ and $v^{-1}\notin I(R)$, and $v$ is a single in $P_{i_1}^{\times},\ldots,P_{i_r}^{\times}$. By Lemma 3.10 $(2)$ of \cite{CSV12}, there is a region $R_v$ obtained from $R$ by switching sides along the $P_{i_j}^{\times}$, and $v^{-1}\in I(R_v)$.  In this case we therefore attach an edge $e_v$ from $x_R$ to $x_{R_v}$. Note that $e_{v^{-1}}=\overl{e}_{v} $.  

Every edge of $(\Sa_{\mathbf{\Pi}})^{(1)}$ carries a \emph{label} which is either some generator $v\in V^{\pm}$ or some partition $\textbf{P}_i$. Two edges $e_{l_1}$, $e_{l_2}$ have \emph{commuting labels} if one of the following holds\begin{enumerate}
\item $l_1=v\in V^{\pm}$, $l_2=w\in V^{\pm}$ and $v,w$ are distinct an commute in $A_\Gamma$,
\item $l_1=v\in V^{\pm}$, $l_2=P_i$ and $v\in lk(P_i)$,
\item $l_1=P_i$, $l_2=P_j$ and $\textbf{P}_i,\textbf{P}_j$ are distinct and commute.  
\end{enumerate}

With commuting labels defined as above, any collection of $k$ edges with commuting labels at a vertex $x_R$ forms the corner of the 1-skeleton of a $k$-cube in  $(\Sa_{\mathbf{\Pi}})^{(1)}$, with parallel edges carrying the same label (\cite{CSV12}, Corollary 3.12). To finish the construction of  $\Sa_{\mathbf{\Pi}}$, we fill in all such $k$-cubes where they occur. $\Sa_{\mathbf{\Pi}}$ is called the \emph{blow-up} of $\Sa_\Gamma$ along $\mathbf{\Pi}$.  We have

\begin{thm}\label{CATBlowup} $($\cite{CSV12}, Theorem 3.14$)$ $\Sa_{\mathbf{\Pi}}$ is connected, locally CAT(0) and $\pi_1(\Sa_{\mathbf{\Pi}})\cong A_\Gamma$.
\end{thm}

\begin{definition}After crossing an edge $e_v$ labelled by a generator $v\in V$, there is a path in $E_\Pi^{(1)}$ connecting the two endpoints of $e_v$.  This path crosses edges labelled by every partition containing $v$ as a singleton.  We call such a path a \emph{characteristic loop} $\gamma_v$.  
\end{definition}

\subsection{Automorphisms of blow-ups}

\begin{proposition} Every blow-up $X$ satifies the hypotheses of Proposition \ref{Criterion}. 
\end{proposition}
\begin{proof} Without loss of generality, $X=\Sa_\Pi$ is a blow-up of the standard Salvetti, hence comes equipped with some labeling of the 1-skeleton by generators $v_1,\ldots,v_n\in V$ or partitions $\textbf{P}_1,\ldots, \textbf{P}_k\in \Pi$.   The hyperplanes of $X$ are in one-to-one correspondence with these labels, so we check them one by one.  To this end, let $l_1$ and $l_2$ be labels with corresponding hyperplanes $K_{l_1}$, $K_{l_2}$.

$K_{l_1}\cap K_{l_2}\neq \emptyset$: Observe that $K_{l_1}$, $K_{l_2}$ intersect if and only if their corresponding labels commute. Consider a square bounded on by edges $e_1$ and $e_2$ dual to $K_{l_1}$ and $K_{l_2}$, respectively.  Each edge $e_i$ can be completed to a characteristic loop $\gamma_i$ as follows. If $l_i=v$ is a generator, then take a characteristic loop $\gamma_v$. If $l_i=P$ is a partition, choose some $m\in max(P)$, and complete this to a characteristic loop $\gamma_m$.  Next observe that since $lk(l_1)\subset st(l)$ for every label $l$ occurring on $\gamma_1$, $l_2$ commutes with every such $l$, and similarly for $\gamma_2$ and $l_1$.  It follows that $\gamma_1\subseteq K_{l_2}$ and $\gamma_2\subseteq K_{l_1}$. Thus, $\gamma_i.K_{l_j}=\delta_{ij}$ and this case is satisfied.  

$K_{l_1}\cap K_{l_2}= \emptyset$ and $K_{l_1}\cup K_{l_2}$ separates: Since $E_{\Pi}$ contains all of the vertices of $\Sa_\Gamma$, it is easy to see that if the $l_i$ both correspond to generators, then $K_{l_1}\cup K_{l_2}$ cannot separate.  Thus the only possibilities for pairs of separating hyperplanes are one generator, one partition or two partitions.  

First suppose $l_1=v$ and $l_2=P$.  We know that $K_P$ disconnects $E_\Pi$ into two components, corresponding to vertices which contain $P$ and those which contain $P^*$. If $K_v\cup K_P$ separate, then we must have $v\in single(P)$, and in fact $\{v\}=single(P)=max(P)$. Then $(P,v)$ is one of the $\Gamma$-Whitehead partitions in $\Pi$.  By assumption, this partition is non-degenerate; hence there must be $w_1^\pm\in double(P)$ and $w_2^\pm\in double(P^*)$.  The hyperplanes corresponding to $w_1$ and $w_2$ do not separate there respective components. 

Now assume $l_1=P$ and $l_2=Q$.  Then $P$ and $Q$ are compatible and do not commute, hence without loss of generality we have $P\subset Q$ and $Q^*\subset P^*$ by Lemma 3.4 of \cite{CSV12}. Then $P^*\cap Q\neq \emptyset$.  In $E_\Pi$, deleting $K_P$ and $K_Q$ leaves three components $E_1$, $E_2$ and $E_3$ whose vertices correspond to regions containing $P\cap Q$, $P^*\cap Q$, and $P^*\cap Q^*$, respectively.  Elements of $single(P)\setminus single(Q)$ connect $E_1$ and $E_2$, elements of $single(Q)\setminus single(P)$ connect $E_2$ and $E_3$, while elements of $single(P)\cap single(Q)$ connect $E_1$ and $E_3$.  If $max(Q)\neq max(P)$ then $K_P\cup K_Q$ does not separate. Then if $max(P)=max(Q)$, the only way  $K_P\cup K_Q$ separates is if actually $single(P)=single(Q)$.  As $P^*\cap Q\neq \emptyset$ we must have that $double(Q)\neq double(P)$, or else $P=Q$, which is impossible.  Then the component containing $E_2$ has a non-separating hyperplane labeled by some $w^\pm\in double(Q)\setminus double(P)$.  If $single(P)$ is not a single element, then the hyperplane corresponding to any element of $single(P)$ does not disconnect the component containing $E_1\cup E_3$. Otherwise, $\{m\}=max(P)=single(P)=single(Q)$ is a single generator.  In this case, since $(P,m)$ is non-trivial, there exists $v^\pm\in double(P)$, and the hyperplane $K_v$ does not separate the component containing $E_1\cup E_3$. 

Finally, to see that condition (3) of Proposition \ref{Criterion} is satisfied, observe that for each maximal collection of pairwise commuting hyperplanes, there is a unique cube in which they all meet.  If $f:X\rightarrow X$ is an automorphism which preserves every hyperplane, $f$ must fix each of these cubes pointwise. Since the union of these cubes covers $X$, $f$ is the identity.  This completes the proof.  
\end{proof}

\begin{corollary} \label{nontrivial}Every automorphism of a blow-up a Salvetti acts nontrivially on $H_1$.
\end{corollary}
\noindent
We are now in a position to finish off the proof of Theorem \ref{torsionfree}:

\noindent
\emph{Proof:} Suppose for contradiction there exists $\phi\neq1\in \mathcal{I}(A_\Gamma)$ such that $\phi^n=1$. Passing to a power, we may assume $n$ is prime.  By Proposition \ref{FixedSal}, there exists a blow-up of a Salvetti $X$ and an automorphism $f:X\rightarrow X$ such that $f_*=\phi\in \Out(A_\Gamma)$. Corollary \ref{nontrivial} now implies that if $f$ acts trivially on $H_1(X)$, $f$ is the identity. Hence, $\phi=1$, a contradiction. \qed


\bibliographystyle{plain}
\bibliography{RaagBib}

\end{document}